\documentclass[12pt]{amsart}
\usepackage{amsmath, amssymb, amsfonts, amsthm, mathtools}
\usepackage{hyperref}
\usepackage{tabularx}
\usepackage{booktabs}
\usepackage{caption}
\usepackage{aurical}
\usepackage{xcolor}

\newtheorem{thm}{Theorem}[section]
\newtheorem{lem}[thm]{Lemma}
\newtheorem{cor}[thm]{Corollary}
\newtheorem{prop}[thm]{Proposition}
\newtheorem{ex}[thm]{Example}

\newtheorem*{prob*}{Open problem}

\theoremstyle{definition}

\newtheorem{defi}[thm]{Definition}

\theoremstyle{remark}

\newtheorem{rem}[thm]{Remark}
\newtheorem*{rem*}{Remark}

\DeclareMathOperator{\rad}{rad}
\DeclareMathOperator{\Hom}{Hom}

\newcommand{\kringel}{\mathbin{\raise0.5pt\hbox{$\scriptstyle\circ$}}}
\newcommand{\pkt}{\mathbin{\raise0.5pt\hbox{$\scriptstyle\bullet$}}}
\newcommand{\sq}{\mathbin{\raise0.5pt\hbox{$\scriptscriptstyle\square$}}}

\newcommand{\C}{\mathbb{C}}

\newcommand{\ad}{{\rm ad}}

\newcommand{\Der}{{\rm Der}}
\newcommand{\Sym}{{\rm Sym}}

\newcommand{\nil}{\mathop{\rm nil}}

\newcommand{\Lf}{\mathfrak{f}}
\newcommand{\Lg}{\mathfrak{g}}

\newcommand{\Ll}{\mathfrak{l}}
\newcommand{\Ln}{\mathfrak{n}}

\newcommand{\Lr}{\mathfrak{r}}
\newcommand{\Ls}{\mathfrak{s}}

\newcommand{\CA}{\mathcal{A}}

\newcommand{\ep}{\varepsilon}

\newcommand{\ra}{\rightarrow}

\renewcommand{\phi}{\varphi}

\begin{document}

\title[Cohomology]{Cohomology of perfect Lie algebras}


\author[D. Burde]{Dietrich Burde}
\author[F. Wagemann]{Friedrich Wagemann}
\address{Fakult\"at f\"ur Mathematik\\
  Universit\"at Wien\\
  Oskar-Morgenstern-Platz 1\\
  1090 Wien \\
  Austria}
\email{dietrich.burde@univie.ac.at}
\address{Laboratoire de math\'ematiques Jean Leray\\
  UMR 6629 du CNRS\\
  Universit\'e de Nantes \\
  2, rue de la Houssini\`ere, F-44322 Nantes Cedex 3 \\
  France}
\email{wagemann@math.univ-nantes.fr}

\date{\today}


\begin{abstract}
We study the adjoint cohomology of perfect Lie algebras over the complex numbers. For the family of perfect Lie algebras
$\Lg=\mathfrak{sl}_2(\C)\ltimes V_m$ we obtain some explicit results for $H^k(\Lg,\Lg)$ with $k\ge 0$.
Here $V_m$ is the irreducible representation of $\mathfrak{sl}_2(\C)$ of dimension $m+1$.
For the computation of the cohomology we use the Hochschild-Serre formula, a long exact sequence in the
cohomology and explicit formulas for the multiplicities of $V_k$ in the exterior product $\Lambda^j(V_m)$ for $j\le 4$.
In general we cannot determine the total adjoint cohomology for $\mathfrak{sl}_2(\C)\ltimes V_m$, but for some small
$m$ this is possible. We also give a classification of complex perfect Lie algebras $\Lg$ of dimension $n\le 9$ and
explicitly compute the cohomology spaces $H^k(\Lg,\Lg)$ with $k=0,1,2$ for all Lie algebras from the classification list.  
\end{abstract}

\maketitle

\section{Introduction}

Let $\Lg$ be a finite-dimensional Lie algebra over a field of characteristic zero. If $\Lg$ is semisimple, then
there are many vanishing results known on the Lie algebra cohomology. For example, by Whitehead's first and second lemma we have
$H^1(\Lg,M)=H^2(\Lg,M)=0$ for every finite-dimensional $\Lg$-module $M$. It is natural, to study the cohomology also for
{\em perfect} Lie algebras, i.e., for Lie algebras $\Lg$ satisfying $[\Lg,\Lg]=\Lg$. It turns out that the Whitehead lemmas
are no longer true for perfect Lie algebras. So the study of the first and second cohomology groups is much
more interesting in this case. \\[0.2cm]
The most important $\Lg$-modules for the coefficients of the cohomology of $\Lg$ are the trivial module and the adjoint module.
We are interested here in particular in the adjoint module. One reason is a conjecture by Pirashvili \cite{P1}, stating that
a perfect complex Lie algebra is semisimple if and only if the adjoint cohomology vanishes, i.e., if it satisfies
$H^n(\Lg,\Lg)=0$ for all $n\ge 0$. One direction is easy to see, namely that the adjoint cohomology vanishes for
complex semisimple (and hence perfect) Lie algebras. However the converse direction is still open. Many authors have constructed
perfect non-semisimple Lie algebras $\Lg$ satisfying 
\[
H^0(\Lg,\Lg)=H^1(\Lg,\Lg)=H^2(\Lg,\Lg)=0,
\]
see \cite{A1,ABB,B1,B2}. But so far, nobody found a perfect non-semisimple Lie algebra with vanishing adjoint cohomology,
or could prove that such a Lie algebra does not exist. \\[0.2cm]
On the other hand, the adjoint cohomology of perfect Lie algebras is interesting in general, and not so much is known about it.
We began a study in \cite{BU75}, which we will continue here. There are some explicit results in the literature for 
particular families, namely for the perfect non-semisimple Lie algebras $\Lg=\mathfrak{sl}_2(\C)\ltimes V_m$, where $V_m$ is the
irreducible representation of $\mathfrak{sl}_2(\C)$ of dimension $m+1$ with $m\ge 1$.
Here the cohomology groups $H^k(\Lg,\Lg)$ have been determined for $k=0,1,2$, see \cite{RAU,RIC} and the references therein. \\[0.2cm]
In this paper we will extend these results, and compute the adjoint cohomology in various other cases.
In particular, for odd $m=2n-1$, we show in Proposition $\ref{3.20}$ that
\[
H^3(\Lg,\Lg) \cong \C^{\lfloor \frac{n+1}{3}\rfloor}, \; H^4(\Lg,\Lg) \cong \C.
\]
Here we use a long exact sequence in cohomology and compute certain spaces determined by the  
multiplicities $N(j,k,n)$ of $V_{jk-2n}$ in the exterior product $\Lambda^j(V_{j+k-1})$ for $j\le 4$.
The numbers $N(j,k,n)$ are given by the coefficient of $q^n$ in the Gaussian $q$-polynomial $(1-q)\genfrac[]{0pt}{1}{j+k}{k}_q$.
We have $N(j,k,n)=p(j,k,n)-p(j,k,n-1)$, where $p(j,k,n)$ is the number of partitions of $n$ into at most $k$ parts
with largest part at most $j$. We are able to compute some of these multiplicities, so that we can
apply the Hochschild-Serre formula to obtain a result on the adjoint cohomology of $\Lg$.
For $m=1,2,3,5,7$ we can even compute the total adjoint cohomology of $\mathfrak{sl}_2(\C)\ltimes V_m$. \\[0.2cm]
Finally, we give a classification of complex perfect Lie algebras $\Lg$ of dimension $n\le 9$ and
explicitly compute the cohomology spaces $H^k(\Lg,\Lg)$ with $k=0,1,2$ for all Lie algebras from the classification list.

\section{Perfect Lie algebras}

Let $\Lg$ be a finite-dimensional complex Lie algebra. Although many results also hold for arbitrary fields of characteristic zero,
it is sufficient for our main results to consider only complex Lie algebras.
We denote by $\Der(\Lg)$ the Lie algebra of derivations of $\Lg$, and by $\ad(\Lg)$ the ideal of inner derivations in $\Der(\Lg)$.
Furthermore, we denote by $Z(\Lg)$ the center of $\Lg$, by $\rad(\Lg)$ the solvable radical
of $\Lg$, and by $\nil(\Lg)$ the nilradical of $\Lg$.

\begin{defi}
A Lie algebra $\Lg$ is called {\em perfect}, if it satisfies $[\Lg,\Lg]=\Lg$. 
\end{defi}

Note that a Lie algebra $\Lg$ is perfect if and only if $H^1(\Lg,\C)=0$ for the trivial $\Lg$-module $\C$.
It is known that the solvable radical of a perfect Lie algebra is nilpotent, see for example Lemma $2.4$ in \cite{BU75}.

\begin{lem}
Let $\Lg$ be a perfect Lie algebra. Then we have $\rad(\Lg)=\nil(\Lg)$.
\end{lem}  

Of course every semisimple Lie algebra is perfect, but the converse need not be true. We can use the Levi decomposition
of a Lie algebra to decide whether or not it is perfect. The following result is well known.

\begin{prop}
Let $\Lg=\Ls\ltimes \rad(\Lg)$ be a Levi decomposition of a Lie algebra $\Lg$ with solvable radical $\Lr=\rad(\Lg)$. Consider 
$V=\Lr/[\Lr,\Lr]$ as an $\Ls$-module. Then $\Lg$ is perfect if and only if $V$ does not contain the trivial
$1$-dimensional $\Ls$-module.
\end{prop}

Using this result we can construct perfect non-semisimple Lie algebras as follows.

\begin{ex}
Let $\Ls$ be a semisimple Lie algebra and $V$ be a representation of $\Ls$ not
containing the $1$-dimensional trivial representation, and define a bracket 
on $\Ls \times V$ by 
\[
[(X,v),(Y,u)] := ([X,Y],Xu-Yv). 
\]
This turns the vector space $\Lg=\Ls \dotplus V$ into a perfect Lie algebra, the semidirect
product $\Lg=\Ls\ltimes V$. 
\end{ex}

Note that $Z(\Lg)=0$ for these Lie algebras, and that $\rad(\Lg) = V$ is abelian. Consider the Lie algebra
$\mathfrak{sl}_2(\C)$ with basis $(e_1,e_2,e_3)$ and Lie brackets given by
\[
[e_1,e_2]=e_3, \, [e_1,e_3] = -2e_1,\, [e_2,e_3] = 2e_2. 
\]
Let $V_m$ be the irreducible representation of dimension $m+1$ with basis $(e_4,\ldots ,e_{m+4})$.

\begin{ex}\label{2.5}
The Lie algebra $\Lg=\mathfrak{sl}_2(\C)\ltimes V_m$ with $m\ge 1$ is a non-semisimple
perfect Lie algebra of dimension $m+4$, with basis $(e_1,\ldots ,e_{m+4})$.
\end{ex}

It is clear that $\Lg$ has a nonzero solvable radical, so that it is not semisimple. Also,
it has a trivial center. However, note that a perfect Lie algebra may have a non-trivial center in general.
For $m=1$ the Lie brackets for $\Lg=\mathfrak{sl}_2(\C)\ltimes V_1$ are given by
\vspace*{0.5cm}
\begin{align*}
[e_1,e_2] & = e_3,     & [e_2,e_3] & = 2e_2,    & [e_3,e_5]& =-e_5.\\
[e_1,e_3] & = -2e_1,   & [e_2,e_4] & = e_5,     &                  \\
[e_1,e_5] & = e_4,     & [e_3,e_4] & = e_4,     &                   \\
\end{align*}

\section{Adjoint Lie algebra cohomology}

In this section we want to obtain results on the adjoint cohomology spaces $H^k(\Lg,\Lg)$ with $k\ge 0$ for
the perfect Lie algebras $\Lg=\mathfrak{sl}_2(\C)\ltimes V_m$ with $m\ge 1$. For $k\le 1$ this is well known.

\begin{prop}\label{3.1}
Let  $\Lg=\mathfrak{sl}_2(\C)\ltimes V_m$. Then we have, for all $m\ge 1$,
\begin{align*}
H^0(\Lg,\Lg) & = 0, \\
H^1(\Lg,\Lg) & \cong \C.
\end{align*}  
\end{prop}

\begin{proof}
For $k=0$ we have $H^0(\Lg,\Lg)=Z(\Lg)$ for all Lie algebras $\Lg$. Hence the claim follows from $Z(\Lg)=0$ for
$\Lg=\mathfrak{sl}_2(\C)\ltimes V_m$. The result for $k=1$ is a special case of a general result, see Proposition
$3.3$ in \cite{BU75}. Let  $\Lg=\Ls\ltimes V$, where $\Ls$ is semisimple and $V$ is an $\Ls$-module. Then we have
\[
H^1(\Lg,\Lg)\cong H^1(V,V)^{\Ls}\cong\Hom_{\Ls}(V,V).
\]
In particular, if $V$ is simple, we have $\dim H^1(\Lg,\Lg)=1$. 
\end{proof}

We recall the following from \cite{BU75}. Let $\Lg=\Ls\ltimes V$, where $\Ls$ is semisimple and $V$ is an $\Ls$-module.
Consider $V$ as an abelian Lie algebra. Then we have a short exact sequence of $\Lg$-modules
\[
0\ra V\ra \Lg\ra \Lg/V\ra 0. 
\]  
This is also a short exact sequence of $V$-modules by restriction to $V\subseteq \Lg$. Here $V$ and $\Lg/V$ are trivial
modules. This yields a long exact sequence in cohomology, because the functor of $\Ls$-invariants is exact on finite-dimensional
modules.

\begin{lem}\label{3.2}
Let $\Lg=\Ls\ltimes V$, where $\Ls$ is semisimple and $V$ is an $\Ls$-module. Then we have
\begin{align*}
0 & \to H^0(V,V)^\Ls\to H^0(V,\Lg)^\Ls\to H^0(V,\Lg/V)^\Ls \\
  & \to H^1(V,V)^\Ls\to H^1(V,\Lg)^\Ls\to H^1(V,\Lg/V)^\Ls \\
  & \to H^2(V,V)^\Ls\to H^2(V,\Lg)^\Ls\to H^2(V,\Lg/V)^\Ls \\
  & \to H^3(V,V)^\Ls\to H^3(V,\Lg)^\Ls\to H^3(V,\Lg/V)^\Ls \to \cdots 
\end{align*}
The sequence continues until $k> \dim V$, in which case we have
\[
H^k(V,V)^{\Ls}=H^k(V,\Lg)^\Ls=H^k(V,\Lg/V)^\Ls=0.
\]  
\end{lem}

For each $k\ge 0$ the following holds, see \cite{BU75}.

\begin{lem}\label{3.3}
Let $\Lg=\Ls\ltimes V$, where $\Ls$ is semisimple and $V$ is an $\Ls$-module. Then we have
\begin{align*}
H^k(V,V)^{\Ls} & \cong \Hom_{\Ls}(\Lambda ^k(V),V),\\  
H^k(V,\Lg/V)^{\Ls} & \cong \Hom_{\Ls}(\Lambda^k(V),\Ls).
\end{align*}  
\end{lem}  

The Hochschild-Serre formula from \cite{HS} yields the following result, see \cite{BU75}.

\begin{prop}\label{3.4}
Let $\Lg=\Ls\ltimes V$ with $\Ls=\mathfrak{sl}_2(\C)$, $V=V_m$ and $k\ge 0$. Then we have  
\begin{align*}
H^k(\Lg,\Lg) & \cong \bigoplus_{i+j=k}H^i(\Ls,\C)\otimes H^j(V,\Lg)^{\Ls}\\[0.1cm]
             &  \cong H^k(V,\Lg)^{\Ls}\oplus H^{k-3}(V,\Lg)^{\Ls}.
\end{align*}
In particular we have
\begin{align*}
H^k(\Lg,\Lg) & \cong H^k(V,\Lg)^{\Ls} \text{ for } k=0,1,2.
\end{align*}
\end{prop}

The idea for computing the cohomology spaces $H^k(\Lg,\Lg)$ is as follows. First we determine
the spaces $H^k(V,V)^{\Ls}$ and $H^k(V,\Lg/V)^{\Ls}$ in Lemma $\ref{3.3}$ by computing the multiplicities
of the irreducible $\Ls$-modules $V$ and $\Ls$ in the exterior product $\Lambda^k(V)$. Then we try to obtain
information on the spaces $H^k(V,\Lg)^{\Ls}$ from the long exact sequence in cohomology in Lemma $\ref{3.2}$.
In the last step we use Proposition $3.4$ to conclude a result on $H^k(\Lg,\Lg)$. \\[0.2cm] 
Let $\Ls=\mathfrak{sl}_2(\C)$. Because of $\Lambda^1(V)\cong V$ we have
\begin{align*}
H^1(V,V)^{\Ls} & \cong \Hom_{\Ls}(V,V)\cong \C, \\
H^1(V,\Lg/V)^{\Ls} & \cong \Hom_{\Ls}(V,\Ls) \cong \begin{cases} \C, \text{ if } V\cong V_2=\Ls \\
0, \text{ otherwise} \end{cases}
\end{align*}
by Schur's Lemma. For $\Lambda^2(V)$ we have the following result.

\begin{lem}\label{3.5}
Let $V_m$ denote the simple $\mathfrak{sl}_2(\C)$-module of dimension $m+1$. Then the decomposition of the
$\mathfrak{sl}_2(\C)$-modules $\Lambda^2(V_{2n-1})$ and $\Lambda^2(V_{2n})$ into simple modules is given by
\begin{align*}
\Lambda^2(V_{2n-1}) & \cong V_0\oplus V_4\oplus V_8\oplus \cdots \oplus V_{4n-4}, \\[0.1cm]
\Lambda^2(V_{2n})  & \cong V_2\oplus V_6\oplus V_{10}\oplus \cdots \oplus V_{4n-2}.
\end{align*}  
\end{lem}

\begin{proof}
Let $V_1$ denote the natural  $\mathfrak{sl}_2(\C)$-module. Then we have  $V_m\cong \Sym ^m(V_1)$, and
by \cite{FUH}, Exercises $11.35$ and $11.31$, 
\begin{align*}
\Lambda^2(\Sym^m(V_1)) & \cong \Sym^2(\Sym^{m-1}(V_1) \\
                     & \cong \bigoplus_{k\ge 0} \Sym^{2(m-1)-4k}(V_1).
\end{align*}
For $m=2n-1$ and $m=2n$ we obtain the claim.
\end{proof}  

\begin{cor}\label{3.6}
Let $\Lg=\Ls\ltimes V$ with $\Ls=\mathfrak{sl}_2(\C)$ and $V=V_m$. Then we have
\begin{align*}
 H^2(V,V)^{\Ls} & \cong \begin{cases} \C, \text{ if } m\equiv 2 \bmod 4 \\
  0, \text{ otherwise } \end{cases} \\
H^2(V,\Lg/V)^{\Ls} & \cong \begin{cases} \C, \text{ if } m\equiv 0\bmod 2 \\
0, \text{ if } m\equiv 1 \bmod 2 \end{cases}
\end{align*}
\end{cor}  

\begin{proof}
If $m=2n-1$ is odd, then
\[
H^2(V,V)^{\Ls}\cong \Hom_{\Ls}(V_0\oplus V_4\oplus \cdots \oplus V_{4n-4},V_{2n-1})=0.
\]  
For $m=2n$ we have
\[
H^2(V,V)^{\Ls}\cong \Hom_{\Ls}(V_2\oplus V_6\oplus \cdots \oplus V_{4n-2},V_{2n}),
\]
which is zero for even $n$, and one-dimensional for odd $n$. The second claim follows similarly.
\end{proof}

In general, there is a formula for the multiplicities of a simple $\mathfrak{sl}_2(\C)$-module $V_{\ell}$ within
$\Sym^ m(V_n)$. It was discovered by Cayley in a letter to Sylvester around $1855$, but only
proved in $1878$ by Sylvester. We use the reference  \cite{HH}, where this formula is
given in Theorem $3.1$ of Section $3$, and is called {\em explicit plethysm}.
Note that the determinant factor there equals one in our case. Furthermore, we have $V_k \cong \Sym(V_1)^k$,
and $\Sym^j(V_k)\simeq \Lambda^j (V_{j+k-1})$. Recall that the Gaussian polynomial
$\genfrac[]{0pt}{1}{j+k}{k}_q$ is defined by
\[
\genfrac[]{0pt}{0}{j+k}{k}_q=\frac{(1-q^{j+k})(1-q^{j+k-1})\cdots (1-q^{j+1})}{(1-q)(1-q^2)\cdots (1-q^k)}.
\]  

\begin{thm}\label{3.7}
There is an isomorphism of $\mathfrak{sl}_2(\C)$-modules
\[
\Lambda^j (V_{j+k-1})\simeq \bigoplus_{n=0}^{\lfloor \frac{jk}{2}\rfloor}(V_{jk-2n})^{\oplus N(j,k,n)},
\]
where $N(j,k,n)$ is the coefficient of $q^n$ in the  Gaussian polynomial $(1-q)\genfrac[]{0pt}{1}{j+k}{k}_q$.
We have
\[
N(j,k,n)=p(j,k,n)-p(j,k,n-1),
\]
where $p(k,j,n)$ is the number of partitions of $n$ into at most $k$ parts, where the largest part is at most $j$.
\end{thm}  

For $j=3$ we obtain the following explicit result.

\begin{prop}\label{3.8}
The multiplicity of the $\mathfrak{sl}_2(\C)$-module $V_2$ in
\[
\Lambda^3(V_{k+2})\simeq \bigoplus_{n=0}^{\lfloor \frac{3k}{2}\rfloor}(V_{3k-2n})^{\oplus N(3,k,n)}
\]
is zero, if $k$ odd. For all even $k\ge 2$ we obtain
\[
N\left(3,k,\frac{3k-2}{2}\right)= \begin{cases} 1, \text{ if $k \equiv 2 \bmod 4$} \\
0, \text{ if $k \equiv 0 \bmod 4$} \end{cases}
\]
\end{prop}  

\begin{proof}
The multiplicity is zero, in case there is no nonnegative integer $n$ with $3k-2n=2$. If $k$ is odd,
then there is no such $n$.
Hence let $k$ be even. Assume first that $k=4r$. Then, by Theorem $1.3$, Formula $(1.8)$ 
in \cite{HH}, and since $N(j,k,n)=N(k,j,n)$,  we have
\begin{align*}
N\left(3,k,\frac{3k-2}{2}\right) & = N(4r,3,6r-1) \\
                & = \left\lfloor \frac{6r-1}{2} \right\rfloor - \left\lfloor \frac{6r-2}{3}\right \rfloor
                  -\left\lfloor \frac{2r-2}{2}\right\rfloor -1 \\
                & = (3r-1)-(2r-1)-(r-1)-1 \\
                & = 0.
\end{align*}
A similar computation for $k=4r+2$ gives
\begin{align*}
N\left(3,k,\frac{3k-2}{2}\right) & = N(4r+2,3,6r+2) \\
  & = (3r+1)-2r-(r-1)-1 \\
  &  =1.
\end{align*}
\end{proof}  

\begin{cor}\label{3.9}
Let $\Lg=\Ls\ltimes V$ with $\Ls=\mathfrak{sl}_2(\C)$ and $V=V_m$ for $m\ge 2$. Then we have
\[
H^3(V,\Lg/V)^{\Ls}\cong \Hom_{\Ls}(\Lambda^3(V),V_2) \cong \begin{cases} \C, \text{ if $m \equiv 0 \bmod 4$} \\
0, \text{ otherwise } \end{cases}
\]
\end{cor}  

\begin{ex}
The decomposition of $\Lambda^3(V_{k+2})$ into simple $\mathfrak{sl}_2(\C)$-modules for $0\le k\le 6$ is given
as follows.
\begin{align*}
\Lambda^3(V_2) &  \cong V_0\\
\Lambda^3(V_3) &  \cong V_3\\
\Lambda^3(V_4) &  \cong V_2\oplus V_6\\
\Lambda^3(V_5) &  \cong V_3\oplus V_5\oplus V_9\\
\Lambda^3(V_6) &  \cong V_0\oplus V_4\oplus V_6\oplus V_8\oplus V_{12}\\
\Lambda^3(V_7) &  \cong V_3\oplus V_5\oplus V_7\oplus V_{11}\oplus V_{13}\oplus V_{15}\\
\Lambda^3(V_8) &  \cong V_2\oplus V_7\oplus V_7\oplus V_8\oplus V_{10}\oplus V_{12}\oplus V_{14}\oplus V_{18}\\
\end{align*}  
\end{ex}

The multiplicity of $V_m$ in $\Lambda^3(V_m)$ has been determined in \cite{BRE}, section $2$.

\begin{prop}\label{3.11}
Let $\Ls=\mathfrak{sl}_2(\C)$ and $V=V_m$ be the irreducible $\Ls$-module of dimension $m+1$. If $m=6q+r$,
with $0\le r\le 5$, then
\[
\dim \Hom_{\Ls}(\Lambda^3(V),V)=  \begin{cases} q, \text{ if $r=0,1,2,4$} \\
q+1, \text{ if $r=3,5$} \end{cases}
\]  
\end{prop}

For $m=2n$ we have $2n=6q+r$, so $r$ can only be $0,2,4$. We obtain $n=3q+s$ with $s=0,1,2$, i.e.,
$q=\lfloor \frac{n}{3}\rfloor$. This gives the following result.

\begin{cor}\label{3.12}
Let $\Ls=\mathfrak{sl}_2(\C)$ and $V=V_m$. Then we have
\[ 
H^3(V,V)^{\Ls}   \cong  \Hom_{\Ls}(\Lambda^3(V),V)\cong \begin{cases} \C^{\lfloor \frac{n}{3}\rfloor},
\hspace*{0.34cm} \text{ if }m=2n \\  
\C^{\lfloor \frac{n+1}{3}\rfloor}, \text{ if } m=2n-1.  \end{cases}
\]
\end{cor}  

We also study the case $j=4$ in Theorem $\ref{3.7}$ in more detail. We obtain a result on the multiplicity of
$V_{2\ell}$ in the exterior product $\Lambda^4(V_{k+3})$ for all $k\ge 1$.

\begin{prop}\label{3.13}
Let $\ell,k \ge 1$. Then the multiplicity of the $\mathfrak{sl}_2(\C)$-module $V_{2\ell}$ in $\Lambda^4(V_{k+3})$ is given by
\[
N(4,k,2k-\ell)=c_{2k-\ell}-c_{2k-2\ell+1},
\]
where the $c_i$ are the coefficients in the power series
\[
C(x)=\frac{1}{(1-x^2)(1-x^3)(1-x^4)}=\sum_{i\ge 0}c_ix^i.
\]
\end{prop}

\begin{proof}
By Theorem $\ref{3.7}$ we have
 \[
\Lambda^4(V_{k+3})\simeq \bigoplus_{n=0}^{2k}(V_{4k-2n})^{\oplus N(4,k,n)}.
\]
For $4k-2n=2\ell$, i.e., for $n=2k-\ell$ we obtain the summand  $(V_{2\ell})^{\oplus N(4,k,2k-\ell)}$. We have
\[
N(4,k,2k-\ell)=p(4,k,2k-\ell)-p(4,k,2k-\ell-1).
\]  
Note that we have the identity
\[
p(j,k,n)=p(j,k-1,n)+p(j-1,k,n-k),
\]  
see Equation $(3.2.6)$ in the proof of Theorem $3.1$ in \cite{AND}. This yields
\[
p(4,k,2k-\ell-1)=p(4,k-1,2k-\ell-1)+p(3,k,k-\ell-1).
\]
The number $p(4,k,2k-\ell)$ counts the partitions into maximal $k$ parts of size at most $4$.
Here $p(4,k-1,2k-\ell-1)$ counts those partitions among them, whose smallest part is equal to one. Hence
\[
p(4,k,2k-\ell)=p(4,k-1,2k-\ell-1)+P(2k-\ell),
\]
where $P(2k-\ell)$ is the number of partitions of $k-\ell-1$ with parts in $\{2,3,4\}$, so not having a part
equal to one. Such a  partition automatically cannot have more than $k$ parts. Together we obtain
\[
N(4,k,2k-\ell)= P(2k-\ell)-p(3,k,k-\ell-1).
\]  
The coefficients $c_m$ in $C(x)$ equal $P(m)$, because the parts are in $\{2,3,4\}$. This yields $P(2k-\ell)=c_{2k-\ell}$.
It remains to show that $p(3,k,k-\ell-1)=c_{2k-2\ell+1}$. It is clear that $p(3,k,k-\ell-1)$ is the coefficient $d_m$ in
\[
D(x)=\frac{1}{(1-x)(1-x^2)(1-x^3)}=\sum_{m\ge 0} d_mx^m
\]  
for $m=k-\ell-1$. Because 
\[
\frac{C(x)-C(-x)}{2}=\frac{x^3}{(1-x^2)(1-x^4)(1-x^6)}
\]
is the series with only odd coefficients $c_{2m+1}$, we see that
\[
\sum_{m\ge 0} c_{2m+3}x^m= \frac{1}{(1-x)(1-x^2)(1-x^3)}  =\sum_{m\ge 0} d_mx^m.
\]
In particular, $d_m=c_{2m+3}$. For $m=k-\ell-1$ this gives
\[
p(3,k,k-\ell-1)=d_{k-\ell-1}=c_{2(k-\ell-1)+3}=c_{2k-2\ell+1}.
\]  
\end{proof}

For $\ell=1$ we have $c_{2k-\ell}-c_{2k-2\ell+1}=0$. This yields the following corollary.

\begin{cor}\label{3.14}
The adjoint $\mathfrak{sl}_2(\C)$-module $V_2$ does not occur in $\Lambda^4(V_{k+3})$ for all $k\ge 1$, i.e., we have
\[
N(4,k,2k-1)=p(4,k,2k-1)-p(4,k,2k-2)=0.
\]
\end{cor}

This result is remarkable. It is also surprising that these partition numbers for $2k-1$ and $2k-2$ are equal.
We can now determine the spaces $H^4(V,\Lg/V)^{\Ls}$ for our family of Lie algebras.

\begin{prop}\label{3.15}
Let $\Lg=\Ls\ltimes V$ with $\Ls=\mathfrak{sl}_2(\C)$ and $V=V_m$ for $m\ge 3$. Then we have
\[
H^4(V,\Lg/V)^{\Ls}=\Hom_{\Ls}(\Lambda^4(V),V_2) = 0.
\]
\end{prop}  

\begin{ex}
Let us denote $V_m\oplus V_m$ by $2V_m$. The decomposition of $\Lambda^4(V_{k+3})$ into simple $\mathfrak{sl}_2(\C)$-modules
for $0\le k\le 5$ is given as follows.
\begin{align*}
\Lambda^4(V_3) &  \cong V_0\\
\Lambda^4(V_4) &  \cong V_4\\
\Lambda^4(V_5) &  \cong V_0\oplus V_4\oplus V_8\\
\Lambda^4(V_6) &  \cong V_0\oplus V_4\oplus V_6\oplus V_8\oplus V_{12}\\
\Lambda^4(V_7) &  \cong V_0\oplus 2V_4\oplus 2V_8\oplus V_{10}\oplus V_{12}\oplus V_{16}\\
\Lambda^4(V_8) &  \cong V_0\oplus 2V_4\oplus V_6\oplus 2V_8\oplus V_{10}\oplus 2V_{12}\oplus V_{14}\oplus V_{16}\oplus V_{20}\\
\end{align*}  
\end{ex}

The multiplicity of $V_m$ in $\Lambda^4(V_m)$ has been determined in \cite{BRE}.

\begin{prop}\label{3.17}
Let $\Ls=\mathfrak{sl}_2(\C)$ and $V=V_m$. Then we have
\[ 
H^4(V,V)^{\Ls}  \cong \Hom_{\Ls}(\Lambda^4(V),V)\cong \begin{cases} 0, \hspace*{0.77cm} \text{ if } m\equiv 1\bmod 2 \\  
\C^{f(m)}, \text{ if } m\equiv 0 \bmod 2.  \end{cases}
\]
Here $f(m)$ is a quadratic function in $m$ depending on $r$, where $m=24q+r$ with $0\le r<24$.
\end{prop}  

Now we are ready to state our results for the adjoint cohomology spaces
$H^k(\Lg,\Lg)$ for the Lie algebra $\Lg=\mathfrak{sl}_2(\C)\ltimes V_m$.

\begin{prop}\label{3.18}
Let  $\Lg=\mathfrak{sl}_2(\C)\ltimes V_m$. Then we have
\[ 
H^2(\Lg,\Lg) = \begin{cases} \C, \text{ if } m\equiv 2\bmod 4, \text{ or } m=4 \\  
0, \text{ otherwise }  \end{cases}
\]
\end{prop}  

\begin{proof}
This follows from \cite{RAU}, Proposition $8$, where the result is formulated in terms
of the dimension of $\Lg$. We can give an independent proof, using a different method, for the case
that $m$ is odd, and for $m=4$. Let $V=V_m$ and $\Ls=\mathfrak{sl}_2(\C)$. 
Then we have, by Lemma $\ref{3.2}$ and
Corollary $\ref{3.6}$ an exact sequence
\[
H^2(V,V)^{\Ls}\ra H^2(V,\Lg)^{\Ls}\ra H^2(V,\Lg/V)^{\Ls},
\]
where the outer terms are equal to zero. Hence, by the Hochschild-Serre formula we have 
$H^2(\Lg,\Lg)\cong H^2(V,\Lg)^{\Ls}=0$. For $m=4n$ we have $H^2(V,V)^{\Ls}=0$ and $H^2(V,\Lg/V)^{\Ls}\cong \C$
by Lemma $\ref{3.6}$. Furthermore, $H^3(V,V)^{\Ls}\cong \C^{\lfloor \frac{m}{6}\rfloor}$ by Corollary $\ref{3.12}$.
So the exact sequence of Lemma $\ref{3.2}$, for the exceptional case $m=4$ yields
\[
0\ra H^2(V,\Lg)^{\Ls}\ra \C\ra 0,
\]
and hence $H^2(\Lg,\Lg)\cong \C$.
\end{proof}

Theorem $\ref{3.7}$ also has the following immediate consequence.

\begin{prop}\label{3.19}
Let  $\Lg=\mathfrak{sl}_2(\C)\ltimes V_{2n-1}$ with $n\ge 1$. Then for all even $k\ge 0$ we have
\begin{align*}  
H^k(V,V)^{\Ls} & \cong \Hom_{\Ls}(\Lambda^k(V),V)=0, \\
H^{k+1}(V,\Lg/V)^{\Ls} & \cong \Hom_{\Ls}(\Lambda^{k+1}(V),\Ls)=0,
\end{align*}
where $\Ls=\mathfrak{sl}_2(\C)$ and $V=V_{2n-1}$.   
It follows that the long exact cohomology sequence of Lemma $\ref{3.2}$
splits into $4$-term exact sequences of the form
\[
0\ra H^k(V,\Lg)^{\Ls} \ra H^k(V,\Lg/V)^{\Ls} \ra H^{k+1}(V,V)^{\Ls}\ra H^{k+1}(V,\Lg)^{\Ls}\ra 0
\]
for all even $k\ge 0$.
\end{prop}  

We can compute all cohomology spaces $H^k(\Lg,\Lg)$ for $0\le k\le 4$ in this case.

\begin{prop}\label{3.20}
Let  $\Lg=\mathfrak{sl}_2(\C)\ltimes V_{2n-1}$ with $n\ge 1$. Then we have
\begin{align*}
H^0(\Lg,\Lg) & = 0, \; H^1(\Lg,\Lg) \cong \C,\; H^2(\Lg,\Lg)  = 0,\\
H^3(\Lg,\Lg) & \cong \C^{\lfloor \frac{n+1}{3}\rfloor}, \; H^4(\Lg,\Lg) \cong \C.
\end{align*}  
\end{prop}

\begin{proof}
The result on $H^k(\Lg,\Lg)$ has been shown for $k=0,1,2$ in Proposition $\ref{3.1}$ and Proposition $\ref{3.18}$ 
We have $H^3(V,V)^{\Ls}\cong  \C^{\lfloor \frac{n+1}{3}\rfloor}$ by Corollary $\ref{3.12}$, and $H^j(V,V)^{\Ls}=0$
for $j=0,2,4$ by Proposition $\ref{3.19}$. In the same way we see that $H^j(V,\Lg/V)^{\Ls}=0$ for $j=1,2,3,4$.
So the above $4$-term exact sequences yield $H^3(V,\Lg)^{\Ls}\cong \C^{\lfloor \frac{n+1}{3}\rfloor}$ and
$H^4(V,\Lg)^{\Ls}=0$. Finally, the Hochschild-Serre formula gives
\begin{align*}
H^3(\Lg,\Lg) & \cong H^3(V,\Lg)^{\Ls}\oplus H^0(V,\Lg)^{\Ls}\cong \C^{\lfloor \frac{n+1}{3}\rfloor},\\
H^4(\Lg,\Lg) & \cong H^4(V,\Lg)^{\Ls} \oplus H^1(V,\Lg)^{\Ls}\cong \C.
\end{align*}  
\end{proof}

\section{The total adjoint cohomology}

We can now compute the total adjoint cohomology of $\Lg=\mathfrak{sl}_2(\C)\ltimes V_m$  for some small $m$, namely
for $m=1,2,3,5,7$. Since $\dim \Lg=m+4$ we have $H^k(\Lg,\Lg)=0$ for all $k> m+4$. Furthermore we have
\[
H^{m+4}(\Lg,\Lg)\cong H_0(\Lg,\Lg)\cong \Lg/[\Lg,\Lg]=0,
\]
since $\Lg$ is perfect. \\[0.2cm]
For $m=1$ we have already computed the cohomology in \cite{BU75},  Proposition $3.6$.

\begin{prop}
Let $\Lg=\mathfrak{sl}_2(\C)\ltimes V_1$ and $k\ge 0$. Then we have
\[
H^k(\Lg,\Lg) \cong H^{k-1}(\mathfrak{sl}_2(\C),\C) 
=\begin{cases} \C, \text{ if } k=1,4 \\
0, \text{ otherwise} \end{cases}
\]
\end{prop} 

\begin{prop}
Let $\Lg=\mathfrak{sl}_2(\C)\ltimes V_2$ and $k\ge 0$.
\[
H^k(\Lg,\Lg) \cong \begin{cases} \C, \text{ if } k=1,2,4,5 \\
0, \text{ otherwise} \end{cases}
\]
\end{prop}

\begin{proof}
Let $\Ls=\mathfrak{sl}_2(\C)$ and $V=V_2$. We have $H^0(\Lg,\Lg)=0$, and $H^1(\Lg,\Lg)\cong H^2(\Lg,\Lg)\cong \C$
by Proposition $\ref{3.1}$ and $\ref{3.18}$.
By Proposition $\ref{3.9}$  we have $H^3(V,V)^{\Ls}=H^3(V,\Lg/V)^{\Ls}=0$. Hence the exact sequence from Lemma $\ref{3.2}$
gives $H^3(V,\Lg)^{\Ls}=0$. Similarly, we obtain $H^0(V,\Lg/V)^{\Ls}=0$. Also $H^j(V,\Lg)^{\Ls}=0$ for $j>3$ because of
$\dim V=3$. So the Hochschild-Serre formula yields
\begin{align*}
H^3(\Lg,\Lg) & = H^3(V,\Lg)^{\Ls} \oplus H^0(V,\Lg)^{\Ls} =0, \\
H^4(\Lg,\Lg) & = H^1(V,\Lg)^{\Ls}\cong \C, \\
H^5(\Lg,\Lg) & = H^2(V,\Lg)^{\Ls}\cong \C.
\end{align*}  
\end{proof}

\begin{prop}
Let $\Lg=\Ls \ltimes V$, with $\Ls=\mathfrak{sl}_2(\C)$, $V=V_3$, and let $k\ge 0$. Then we have
\[
H^k(\Lg,\Lg) \cong \begin{cases} \C, \text{ if } k=1,3,4,6 \\
0, \text{ otherwise} \end{cases}
\]
\end{prop}

\begin{proof}
The result for $H^k(\Lg,\Lg)$ with $0\le k\le 4$ follows from Proposition $\ref{3.20}$. By the Hochschild-Serre
formula we then obtain 
\begin{align*}
H^5(\Lg,\Lg) & \cong H^2(V,\Lg)^{\Ls} =0, \\
H^6(\Lg,\Lg) & \cong H^3(V,\Lg)^{\Ls} \cong \C.
\end{align*}
\end{proof}

\begin{prop}
Let $\Lg=\Ls \ltimes V$, with $\Ls=\mathfrak{sl}_2(\C)$, $V=V_5$, and let $k\ge 0$. Then we have
\[
H^k(\Lg,\Lg) \cong \begin{cases} \C, \text{ if } k=1,3,4,5,6,8 \\
0, \text{ otherwise} \end{cases}
\]
\end{prop}

\begin{proof}
We only need to show the result for $k\ge 5$. 
We have $H^{6-j}(V,V)^{\Ls}\cong H^j(V,V)^{\Ls}$ for $j\le 6$, because of $\dim V=6$ and Poincar\'e duality. So we have
\begin{align*}
H^4(V,V)^{\Ls} & \cong H^2(V,V)^{\Ls}=0, \\
H^5(V,V)^{\Ls} & \cong H^1(V,V)^{\Ls}\cong \C.
\end{align*}
By Proposition $\ref{3.19}$ we have $H^4(V,\Lg/V)^{\Ls}=0$.
Altogether we have
\begin{align*}
H^k(V,V)^{\Ls} & \cong  \begin{cases} \C, \text{ if } k=1,3,5 \\
0, \text{ if } k=0,2,4 \end{cases}
\end{align*}
and $H^k(V,\Lg/V)^{\Ls} = 0$ for all $1\le k\le 6$. 
Hence by the $4$-term exact sequences of Proposition $\ref{3.19}$ we have
\begin{align*}
H^k(V,\Lg)^{\Ls} & \cong  \begin{cases} \C, \text{ if } k=1,3,5 \\
0, \text{ if } k=0,2,4 \end{cases}
\end{align*}
So the Hochschild-Serre formula gives
\begin{align*}
H^5(\Lg,\Lg) & \cong H^5(V,\Lg)^{\Ls} \oplus H^2(V,\Lg)^{\Ls} \cong \C, \\
H^6(\Lg,\Lg) & \cong H^6(V,\Lg)^{\Ls} \oplus H^3(V,\Lg)^{\Ls} \cong \C, \\
H^7(\Lg,\Lg) & \cong H^4(V,\Lg)^{\Ls} =0,\\
H^8(\Lg,\Lg) & \cong H^5(V,\Lg)^{\Ls} \cong \C.
\end{align*}
\end{proof} 

\begin{prop}\label{4.5}
Let $\Lg=\Ls \ltimes V$, with $\Ls=\mathfrak{sl}_2(\C)$, $V=V_7$, and let $k\ge 0$. Then we have
\[
H^k(\Lg,\Lg) \cong \begin{cases} \C, \text{ if } k=1,3,4,5,6,7,8,10 \\
0, \text{ otherwise} \end{cases}
\]
\end{prop}

\begin{proof}
We may assume $k\ge 5$ as before. Using the results of section $3$ we obtain
\begin{align*}
H^k(V,V)^{\Ls} & \cong  \begin{cases} \C, \text{ if } k=1,3,5,7 \\
0, \text{ if } k=0,2,4,6 \end{cases}
\end{align*}
and $H^k(V,\Lg/V)^{\Ls} = 0$ for all $1\le k\le 7$. The long exact sequence in the cohomology then implies that
\begin{align*}
H^k(V,\Lg)^{\Ls} & \cong  \begin{cases} \C, \text{ if } k=1,3,5,7 \\
0, \text{ if } k=0,2,4,6 \end{cases}
\end{align*}
So the  Hochschild-Serre formula gives
\begin{align*}
H^5(\Lg,\Lg) & \cong H^5(V,\Lg)^{\Ls} \oplus H^2(V,\Lg)^{\Ls} \cong \C, \\
H^6(\Lg,\Lg) & \cong H^6(V,\Lg)^{\Ls} \oplus H^3(V,\Lg)^{\Ls} \cong \C, \\
H^7(\Lg,\Lg) & \cong H^7(V,\Lg)^{\Ls} \oplus H^4(V,\Lg)^{\Ls} \cong \C,\\
H^8(\Lg,\Lg) & \cong H^5(V,\Lg)^{\Ls} \cong \C, \\
H^9(\Lg,\Lg) & \cong H^6(V,\Lg)^{\Ls} =0, \\
H^{10}(\Lg,\Lg) & \cong H^7(V,\Lg)^{\Ls} \cong \C. 
\end{align*}
\end{proof}

\begin{rem}
For $m=9$ and $\Lg=\mathfrak{sl}_2(\C)\ltimes V_9$ the result is different, since
$H^3(\Lg,\Lg)\cong \C^2$ in this case by Proposition $\ref{3.20}$. In fact, the spaces
$H^k(V,V)^{\Ls}$ for $V=V_m$, $\Ls=\mathfrak{sl}_2(\C)$ with $m=2n-1$ and odd $k\ge 3$ will grow rapidly.
\end{rem}

\section{Adjoint cohomology for Lie algebras of low dimension}

In this section we will compute the adjoint cohomology spaces $H^k(\Lg,\Lg)$ with $k=0,1,2$ for all
complex non-semisimple perfect Lie algebras of dimension $n\le 9$. \\[0.2cm]
The Levi decomposition of such algebras is given by $\Ls\ltimes \Lr$, where $\Ls$ is isomorphic
to either $\mathfrak{sl}_2(\C)$ or $\mathfrak{sl}_2(\C)\oplus \mathfrak{sl}_2(\C)$, and
$\Lr$ is the (non-trivial) solvable radical. \\[0.2cm]
Let $V$ be a $\mathfrak{sl}_2(\C)$-module. By Weyl's theorem, $V$ is a direct sum of simple modules $V_j$. Writing
\[
V_j^{\oplus e}=V_j\oplus \cdots \oplus V_j
\]
for the $e$-fold direct sum, we have 
\[
V \cong V_{n_1}^{\oplus e_1}\oplus \cdots \oplus V_{n_m}^{\oplus e_m},
\]
where $n_1,\ldots ,n_m \ge 1$ are pairwise distinct integers. Note that $V_0$ cannot appear
as a summand, since we assume that $\Lg$ is perfect. We have the following result.

\begin{prop}\label{5.1}
Let $\Lg=\mathfrak{sl}_2(\C)\ltimes V$, where $V$ is a direct sum of $V_j^{\oplus e_j}$ for $j=1,\ldots ,m$. 
Then $\Lg$ is perfect and we have 
\[
\dim  H^1(\Lg,\Lg)=e_1^2+\cdots +e_m^2.
\]
\end{prop}

\begin{proof}
We have
\[  
H^1(\Lg,\Lg)\cong\Hom_{\mathfrak{sl}_2(\C)}(V_{n_1}^{\oplus e_1}\oplus \cdots \oplus V_{n_m}^{\oplus e_m},
V_{n_1}^{\oplus e_1}\oplus \cdots \oplus V_{n_k}^{\oplus e_m}).
\]
Since $V_j$ is a simple $\mathfrak{sl}_2(\C)$-module, Schur's Lemma implies that this space equals
\[
M_{e_1}(\C)\oplus \cdots \oplus M_{e_m}(\C),
\]
which has dimension $e_1^2+\cdots +e_m^2$.
\end{proof}

Perfect Lie algebras also can have a solvable radical, which is not abelian. Let us introduce an example.

\begin{ex}
Let $\mathfrak{sl}_2(\C)={\rm span}\{ e_1,e_2,e_3\}$, and $\mathfrak{n}_3(\C)={\rm span}\{ e_4,e_5,e_6\}$ be the
Heisenberg Lie algebra, with $[e_4,e_5]=e_6$. Consider the following semidirect product of $\mathfrak{sl}_2(\C)$ and
$\Ln_3(\C)$, given by the following Lie brackets in the basis  $(e_1,\ldots ,e_6)$:
\vspace*{0.5cm}
\begin{align*}
[e_1,e_2] & = e_3,     & [e_2,e_3] & = 2e_2,    & [e_3,e_5]& =-e_5,\\
[e_1,e_3] & = -2e_1,   & [e_2,e_4] & = e_5,     & [e_4,e_5] & = e_6.\\
[e_1,e_5] & = e_4,     & [e_3,e_4] & = e_4,     &                   \\
\end{align*}
We denote this Lie algebra by $\mathfrak{sl}_2(\C) \ltimes \Ln_3(\C)$.
\end{ex}

This Lie algebra is perfect, non-semisimple and has a $1$-dimensional center 
$Z(\Lg)=\langle e_6\rangle$. One can generalize
this example to semidirect products $\Ls\ltimes \Ln_{2n+1}(\C)$ for a semisimple Lie algebra 
$\Ls$ and the Heisenberg Lie algebra $\Ln_{2n+1}(\C)$ of dimension $2n+1$.

\begin{prop}
Let $\Lg$ be a complex non-semisimple perfect Lie algebra of dimension $n\le 9$. Then $\Lg$ is isomorphic to
one of the following algebras, with the adjoint cohomology spaces $H^k(\Lg,\Lg)$ for $k=0,1,2$ given as follows:
\vspace*{0.5cm}
\begin{center}
\begin{tabular}{c|ccccc}
\color{red}{$\Lg$} & \color{blue}{$\dim \Lg$} & \color{green}{Turkowski} & \color{purple}{$\dim H^0(\Lg,\Lg)$}
& \color{orange}{$\dim H^1(\Lg,\Lg)$}  & \color{brown}{$\dim H^2(\Lg,\Lg)$} \\[4pt]
\hline
$\Ls\Ll_2(\C)\ltimes V_1$ & $5$ & $L_{5,1}$  & $0$ & $1$ & $0$ \\[4pt]
$\Ls\Ll_2(\C)\ltimes V_2$ & $6$ & $L_{6,4}\cong L_{6,1}$ & $0$ & $1$ & $1$ \\[4pt]
$\Ls\Ll_2(\C)\ltimes \Ln_3(\C)$  & $6$ & $L_{6,2}$  & $1$ & $1$ & $0$ \\[4pt]
$\Ls\Ll_2(\C)\ltimes V_3$ & $7$ & $L_{7,6}$  & $0$ & $1$ & $0$ \\[4pt]
$\Ls\Ll_2(\C)\ltimes (V_1\oplus V_1)$ & $7$ & $L_{7,7}$  & $0$ & $4$ & $0$ \\[4pt]
$\Ls\Ll_2(\C)\oplus (\Ls\Ll_2(\C)\ltimes V_1)$ & $8$ & $\Ls\Ll_2\oplus L_{5,1}$ & $0$ & $1$ & $0$ \\[4pt]
$\Ls\Ll_2(\C)\ltimes V_4$ & $8$ & $L_{8,21}$  & $0$ & $1$ & $1$ \\[4pt]
$\Ls\Ll_2(\C)\ltimes (V_1\oplus V_2)$ & $8$ & $L_{8,22}$  & $0$ & $2$ & $1$ \\[4pt]
$\Ls\Ll_2(\C)\ltimes (V_1\oplus \Ln_3(\C))$ & $8$ & $\color{red}{L_{8,13}^{\ep=0}}$  & $1$ & $3$ & $0$ \\[4pt]
$\Ls\Ll_2(\C)\ltimes \Lf_{2,3}(\C)$ & $8$ & $L_{8,15}$  & $0$ & $1$ & $1$ \\[4pt]
$\Ls\Ll_2(\C)\ltimes_{\phi} \Ln_5(\C)$ & $8$ & $L_{8,13}^1\cong L_{8,13}^{-1}$  & $1$ & $2$ & $1$ \\[4pt]
$\Ls\Ll_2(\C)\ltimes_{\psi} \Ln_5(\C)$ & $8$ & $L_{8,19}$  & $1$ & $1$ & $0$ \\[4pt]
$\Ls\Ll_2(\C)\oplus (\Ls\Ll_2(\C)\ltimes V_2)$ & $9$ & $\Ls\Ll_2\oplus L_{6,1}$  & $0$
& $1$ & $1$ \\[4pt]
$\Ls\Ll_2(\C)\oplus (\Ls\Ll_2(\C)\ltimes \Ln_3(\C))$  & $9$ & $\Ls\Ll_2\oplus L_{6,2}$ & $1$
& $1$ & $0$ \\[4pt]
$\Ls\Ll_2(\C)\ltimes V_5$ & $9$ & $L_{9,59}$  & $0$ & $1$ & $0$ \\[4pt]
$\Ls\Ll_2(\C)\ltimes (V_1\oplus V_3)$ & $9$ & $L_{9,60}$  & $0$ & $2$ & $0$\\[4pt]
$\Ls\Ll_2(\C)\ltimes (V_2\oplus V_2)$ & $9$ & $L_{9,61}$  & $0$ & $4$ & $4$ \\[4pt]
$\Ls\Ll_2(\C)\ltimes (V_1\oplus V_1\oplus V_1)$ & $9$ & $L_{9,63}$  & $0$ & $9$ & $0$ \\[4pt]
$\Ls\Ll_2(\C)\ltimes (V_2\oplus \Ln_3(\C))$ & $9$ & $L_{9,58}$  & $1$ & $2$ & $0$ \\[4pt]
$\Ls\Ll_2(\C)\ltimes (\Ln_3(\C)\oplus \Ln_3(\C))$ & $9$ & $L_{9,37}\cong L_{9,42}$  & $2$
& $2$ & $0$ \\[4pt]
$\Ls\Ll_2(\C)\ltimes \Lf_{3,2}(\C)$ & $9$ & $L_{9,62}$  & $0$ & $2$ & $2$ \\[4pt]
$\Ls\Ll_2(\C)\ltimes \CA_{6,4}(\C)$ & $9$ & $L_{9,41}$  & $2$ & $3$ & $1$ \\[4pt]
\end{tabular}
\end{center}
\end{prop}

\begin{proof}
The classification of perfect Lie algebras in low dimension is based on the the work of  
Turkowski, who classified Lie algebras with non-trivial Levi decomposition up to dimension $8$ over the real
numbers in \cite{TUR1}, and for dimension $9$ over real and complex numbers in \cite{TUR2}. He lists explicit Lie brackets
for all algebras. From this work it is not difficult to derive a classification of complex perfect Lie algebras of
dimension $n\le 9$. We need to add one Lie algebra though, which Turkowski has not in his list. It is the complexification of the
algebra $L_{8,13}^{\ep}$ for $\ep=0$, isomorphic to $\Ls\Ll_2(\C)\ltimes (V_1\oplus \Ln_3(\C))$. Turkowski only allows $\ep=\pm 1$
for $L_{8,13}$. We have computed the adjoint cohomology groups $H^k(\Lg,\Lg)$ for $k=0,1,2$ in each case by computer, and have
compared the results for the cases covered by Proposition $\ref{5.1}$.
Note that $H^0(\Lg,\Lg)=Z(\Lg)$ and $H^1(\Lg,\Lg)=\Der(\Lg)/\ad(\Lg)$.
\end{proof}

\section*{Acknowledgments}
We thank Anton Mellit and Michael Schlosser for helpful discussion. The first author was 
supported by the Austrian Science Foun\-da\-tion (FWF), Grant DOI 10.55776/P33811. 
For open access purposes, the authors have applied a 
CC BY 4.0 public copyright license to any author-accepted manuscript version arising from 
this submission.


\begin{thebibliography}{99}


\bibitem{AND}  G. E. Andrews: {\em The Theory of Partitions}.
Cambridge University Press, Cambridge, 1998, xvi+255 pp.
  
\bibitem{A1} E. Angelopoulos, S. Benayadi: {\em Construction d'alg\`ebres de Lie sympathiques non semi-simples munies de 
produits scalaires invariants}. 
C.\ R.\ Acad.\ Sci.\ Paris, t. \textbf{317}, S\'erie I (1993), 741--744.   

\bibitem{ABB} D. Arnal, H. Benamor, S. Benayadi, G. Pinczon:{\em Une alg\`ebre de Lie non semi-simple rigide et sympathique:
$H^1(\Lg)=H^2(\Lg)=H^0(\Lg,\Lg)=H^1(\Lg,\Lg)=H^2(\Lg,\Lg)=0$}.
C.\ R.\ Acad.\ Sci.\ Paris, t. \textbf{315}, S\'erie I (1992), 261--263.   

\bibitem{B1} S. Benayadi: {\em Certaines propri\'et\'es d'une classe d'alg\`ebres de Lie qui g\'en\'eralisent 
les alg\`ebres de Lie semi-simples}.
Ann.\ Fac.\ Sci.\ Toulouse, $5^e$ s\'erie, t. \textbf{12}, 1 (1991), 29--35.

\bibitem{B2} S. Benayadi: {\em Structure of perfect Lie algebras without center and outer derivations}.
Ann.\ Fac.\ Sci.\ Toulouse, $6^e$ s\'erie, t. \textbf{5}, 2 (1996), 203--231. 

\bibitem{BRE} M. Bremner, H. Elgendy: {\em Alternating quaternary algebra structures on irreducible representations of
$\mathfrak{sl}_2(\C)$}.  Linear Algebra Appl.\ \textbf{433} (2010), no. 8-10, 1686--1705.   

\bibitem{BU75} D. Burde, F. Wagemann: {\em Sympathetic Lie algebras and adjoint cohomology for
Lie algebras}.
Journal of Algebra, Vol. \textbf{629} (2023), 381--398.

\bibitem{FUH} W. Fulton, J. Harris: {\em Representation theory. A first course}.
Graduate Texts in Mathematics, \textbf{129}. Readings in Mathematics. Springer-Verlag, New York, 1991. xvi+551 pp.

\bibitem{HH} H. Hahn, J. Huh, E. Lim, J. Sohn: {\em From partition identities to a combinatorial approach to explicit
 Satake inversion}.
Ann.\ Comb.\ \textbf{22} (2018), no. 3, 543--562.

\bibitem{HS} G. Hochschild, J-P. Serre: {\em Cohomology of Lie algebras}.
Ann.\ Math.\ (2) \textbf{57} (1953), no. 3, 591--603.

\bibitem{P1} T. Pirashvili: {\em On strongly perfect Lie algebras}. 
Communications in Algebra \textbf{41} (2013), no. 5, 1619--1625.

\bibitem{RAU} G. Rauch: {\em Effacement et d\'eformation}.
Ann.\ Inst.\ Fourier (Grenoble) \textbf{22} (1972), no. 1, 239--269.

\bibitem{RIC} R. W. Richardson: {\em On the rigidity of semi-direct products of Lie algebras}.
Pacific J.\ Math.\ \textbf{22} (1967), 339--344.

\bibitem{TUR1} P. Turkowski: {\em Low-dimensional real Lie algebras}.
J.\ Math.\ Phys.\ \textbf{29} (1988), no. 10, 2139--2144.

\bibitem{TUR2} P. Turkowski: {\em Structure of real Lie algebras}.
Linear Algebra and its Applications \textbf{171} (1992), 197--212.

\end{thebibliography}
\end{document}